\documentclass[a4paper,twoside,leqno,11pt]{article}
\usepackage{mathrsfs}
\usepackage{bbm,pifont,latexsym}
\usepackage{anysize}  
\marginsize{2.5cm}{2.5cm}{2.5cm}{2.5cm}

\usepackage{dcolumn,indentfirst,cite,color,hyperref}
\usepackage{amsmath,amssymb,amscd,amsthm,amsfonts,mathrsfs}
\usepackage{color,graphicx,xcolor,graphics}
\usepackage{titlesec} 
\usepackage{titletoc} 
\titlecontents{section}[20pt]{\addvspace{2pt}\filright}
              {\contentspush{\thecontentslabel. }}
              {}{\titlerule*[8pt]{}\contentspage}

\usepackage{fancyhdr} 
\newtheorem{thm}{Theorem}[section]
\newtheorem{lema}[thm]{Lemma}
\newtheorem{cor}[thm]{Corollary}

\newtheorem{defi}[thm]{Definition}

\newcommand{\DF}[2]{{\displaystyle\frac{#1}{#2}}}

\makeatletter\@addtoreset{equation}{section}\makeatother 

\titleformat{\section}{\centering\normalsize}{\textsc{\thesection.}}{1em}{\textsc}

\allowdisplaybreaks 

\begin{document}

\title{\bf\large  THE HAUSDORFF DIMENSION OF THE BOUNDARY OF THE
IMMEDIATE BASIN OF INFINITY OF McMULLEN MAPS
\author{\normalsize{FEI YANG ~AND~ XIAOGUANG WANG}}}
\date{}
\maketitle


\begin{center}
\begin{minipage}{13cm}{\small{\textsc{Abstract.}}}{\small{
 In this paper, we give a formula of the Hausdorff dimension of the boundary of the immediate basin of
 infinity of McMullen maps $f_p(z)=z^Q+p/z^Q$, where $Q\geq 3$ and
 $p$ is small.
 This gives a lower bound of the Hausdorff dimension of the Julia
 sets of McMullen maps in the special cases.
 } }
\end{minipage}
\end{center}


\section{Introduction}

The dynamics of McMullen maps
\begin{equation*}
f_p(z)=z^Q+p/z^Q
\end{equation*}
with $Q\geq 3$ have been studied a lot (\cite{Dev,DG,DLU,St,WQY}).
These special rational maps can be viewed as a perturbation of the
simple polynomial $f_0(z)=z^Q$.

It is known from \cite{Dev,McM} that for small $p$, the Julia set
$J(f_p)$ of $f_p$ consists of uncountably many Jordan curves about
the origin. This kind of Julia set is homeomorphic to $\mathcal
{C}\times \mathbb{S}$, where $\mathcal {C}$ is the middle third
Cantor set and $\mathbb{S}$ is the unit circle (See Figure 1). These
Julia sets are called \textit{Cantor circles}. In this case, all
Fatou components are attracted by $\infty$. We denote by $B_p$ the
immediate attracting basin of $\infty$, then the boundary $\partial
B_p$ is a Jordan curve (actually quasicircle by Lemma
\ref{holo-motion}). In fact, it is proven in \cite{WQY} that
$\partial B_p$ is always a Jordan curve if $J(f_p)$ is not a Cantor
set. In this paper, we obtain the following main theorem:

\begin{thm}\label{main}
Let $Q\geq 3$, then for small $p$ such that $J(f_p)$ is a Cantor
circle, the Haudorff dimension of $\partial B_p$ is
\begin{equation}\label{dim_formula}
\textup{dim}_H(\partial B_p)=1+\DF{|p|^2}{\log
Q}+\mathcal{O}(|p|^3).
\end{equation}
In particular, if $Q\neq 4$, then the higher order
$\mathcal{O}(|p|^3)$ can be replaced by $\mathcal{O}(|p|^4)$.
\end{thm}

As an immediate corollary, the main theorem gives a lower bound of
the Hausdorff dimension of $J(f_p)$ with small $p$.

\begin{figure}[!htpb]
  \setlength{\unitlength}{1mm}
  \centering
  \includegraphics[width=60mm]{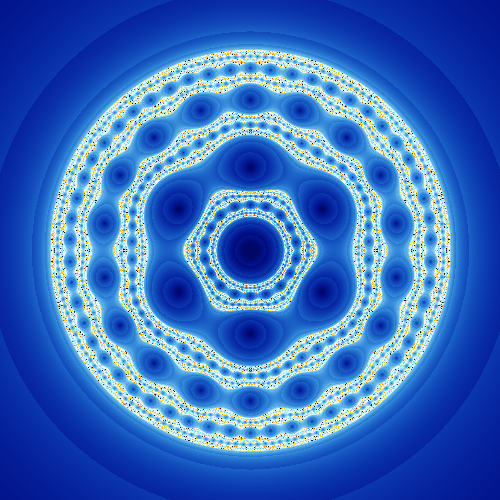}
  \includegraphics[width=60mm]{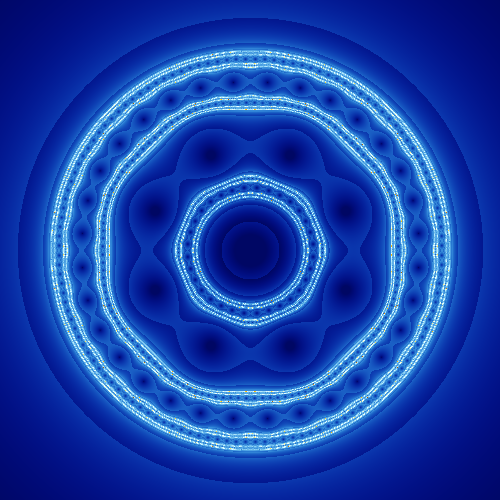}
  \begin{minipage}{14cm}
  \caption{ The Julia sets of $f_p(z)=z^Q+p/z^Q$, where $p=0.005$ and $Q=3,4$
  respectively.
  Both are Cantor circles. Figure range: $[-1.25,1.25]\times[-1.25,1.25]$.}
  \label{Fig_McMullen}
  \end{minipage}
\end{figure}

We would like to mention that for the polynomials $P_c(z)=z^d+c$
with $d\geq 2$ and small $c$ such that $P_c$ is hyperbolic, the
Hausdorff dimension of the Julia set of $P_c$ has been calculated
in\cite{Rue},\cite{WBK} and\cite{CDM}, where the dimensional formula
was expanded to the second order, third order and fourth order in
$c$, respectively. In theory, terms of higher orders can be
calculated successively. However, the calculation become more
complicated as the rising of order.

\section{Proof of the Main Theorem}
The proof of the main theorem is similar to that in\cite{Osb}
and\cite{WBK}. All details of complicated calculations will be
included in the next section. In the following, we always assume
that $p$ is small\,($p=0$ is allowed).

Unlike the polynomials $P_c(z)=z^d+c$, the parameter space of
McMullen family has a special point at $p=0$. The whole Julia set
$J(f_p)$ does not converge to $J(f_0)$ (the unit circle
$\mathbb{S}$) in Hausdorff topology when $p$ tends to $0$, see
\cite{DG}. However,
 the boundary of the immediately attracting basin of infinity
$\partial B_p$ does. In fact, we can show (Lemma \ref{holo-motion})
that  $\partial B_p$ is  a holomorphic motion of the unit circle
$\mathbb{S}$. For this, we first recall the definition of
holomorphic motion.

\begin{defi}[{Holomorphic Motion,\cite{MSS}}]\label{Def_Holo_motion}
\textup{Let $E$ be a subset of $\widehat{\mathbb{C}}$, a map
$h:\mathbb{D}\times E\rightarrow\widehat{\mathbb{C}}$ is called a
\textit{holomorphic motion} of $E$ parameterized by $\mathbb{D}$ and
with base point $0$ if}

\textup{(1) For every $z\in E$, $\beta\mapsto h(\beta,z)$ is
holomorphic for $\beta$ in $\mathbb{D}$;}

\textup{(2) For every $\beta\in\mathbb{D}$, $z\mapsto h(\beta,z)$ is
injective on $E$; and}

\textup{(3) $h(0,z)=z$ for all $z\in E$.}
\end{defi}
The unit disk $\mathbb{D}$ in Definition \ref{Def_Holo_motion} can
be replaced by any other topological disk.

\begin{thm}[{The $\lambda$-Lemma,\cite{MSS}}]\label{e-lemma}
A holomorphic motion $h:\mathbb{D}\times
E\rightarrow\widehat{\mathbb{C}}$ of $E$ has a unique extension to a
holomorphic motion $h:\mathbb{D}\times
\overline{E}\rightarrow\widehat{\mathbb{C}}$ of $\overline{E}$. The
extension is a continuous map. For each $\beta\in \mathbb{D}$, the
map $h(\beta,\cdot):E\rightarrow\widehat{\mathbb{C}}$ extends to a
quasiconformal map of the sphere to itself.
\end{thm}

It's known from \cite{St} that in the parameter space of $f_p$, the
McMullen domain $\mathcal{M}:=\{p\in \mathbb{C}-\{0\}; J(f_p) \text{
is a Cantor circle} \}$ is a deleted neighborhood of the origin. It
turns out that $\mathcal{V}=\mathcal{M}\cup\{0\}$ is a topological
disk containing $0$.

\begin{lema}\label{holo-motion} There is a holomorphic motion $H:\mathcal{V}\times \mathbb{S}\rightarrow\mathbb{C}$
parameterized by $\mathcal{V}$ and with base point $0$ such that
$H(p, \mathbb{S})=\partial B_p$ for all $p\in \mathcal{V}$.
\end{lema}
\begin{proof}
We first prove that every repelling periodic point  of $f_0(z)=z^Q$
moves holomorphically in $\mathcal{V}$.  Let $z_0\in\mathbb{S}$ be
such a point with period $k$. For small $p$, the map $f_p$ is a
small perturbation of $f_0$. By implicit function theorem, there is
a neighborhood $U_0$ of $0$ such that $z_0$ becomes a repelling
point $z_p$ of $f_p$ with the same period $k$, for all $p\in U_0$.
On the other hand, for all $p\in\mathcal{M}$, since $f_p$ has no
non-repelling cycles, each repelling cycle of $f_p$ moves
holomorphically throughout $\mathcal{M}$ (See Theorem 4.2 in
\cite{McM1}).

Since $\mathcal{V}$ is simply connected, there is a holomorphic map
$Z:\mathcal{V}\rightarrow \mathbb{C}$ such that $Z(p)=z_p$ for $p\in
U_0$. Let $\textup{Fix}{(f_0)}$ be all repelling points of $f_0$.
Then the map $h:\mathcal{V}\times \textup{Fix}{(f_0)}\rightarrow
\mathbb{C}$ defined by $h(p,z_0)=Z(p)$ is a holomorphic motion.
Notice that $\mathbb{S}=\overline{\textup{Fix}{(f_0)}}$, by Theorem
\ref{e-lemma}, there is an extension of $h$, say
$H:\mathcal{V}\times \mathbb{S}\rightarrow\mathbb{C}$. It's obvious
that $H(p, \mathbb{S})$ is a connected component of $J(f_p)$.

To finish, we show $H(p, \mathbb{S})=\partial B_p$ for all $p\in
\mathcal{V}$. By the uniqueness of the holomorphic motion of
hyperbolic Julia sets, it suffices to show $H(p,
\mathbb{S})=\partial B_p$ for small and real parameter $p\in
(0,\epsilon)$, where $\epsilon>0$.

Under the small perturbation $f_p$ with $p\in(0,\epsilon)$, the
fixed point $z_0=1$ of $f_0$ becomes the repelling fixed points
$z_p$ of $f_p$, which is real and close to $1$. The map $f_p$ has
exactly two real and positive fixed points. One is $z_p$ and the
other is $z^*_p$, which is near $0$.  It's obvious that $z_{p}$ is
the landing point of the zero external ray of $f_p$. So
$H(p,1)=z_{p}\in\partial B_p$. This implies $H(p,
\mathbb{S})=\partial B_p$ for all $p\in (0,\epsilon)$.
\end{proof}

The boundary $\partial B_p$ is a `repeller' of the map $f_p$ 
in the sense of Ruelle \cite{Rue}.

\begin{thm}[{Ruelle,\cite{Rue}}]\label{Rulle}
If the repeller $J_\lambda$ of a family of real analytic conformal
maps $f_\lambda$ depends analytically on $\lambda$, then the
Hausdorff dimension of $J_\lambda$ depends real analytically on
$\lambda$.
\end{thm}

We define a function $H:\mathcal{V}\rightarrow \mathbb{R}^+$ by
$H(p)=\textup{dim}_H(\partial B_p)$. We first derive some basic
properties of $H$. The fact $\partial B_0=\mathbb{S}$ implies
$H(0)=1$. By Ruelle's theorem, we know that $H$ is a real analytic
function. Thus when $p$ is near $0$, we have
\begin{equation}
H(p)=\sum_{s,t\geq0}a_{st}p^s\overline{p}^t, \ a_{00}=1.
\end{equation}

 It
follows from  $\overline{f_p(\bar{z})}=f_{\bar{p}}(z)$ and $
f^{\circ 2}_{{e}^{2\pi i/(Q-1)}p}({e}^{\pi i/(Q-1)} z)=
 {e}^{\pi i/(Q-1)} f^{\circ 2}_{p}(z)$  that
 $$\overline{H(p)}=H(p)=H(\bar{p}),\ \  H({e}^{2\pi i/(Q-1)}p)=H(p).$$
So the coefficients satisfy
$$a_{st}=\overline{a_{st}}=a_{ts}, \ a_{st}=a_{st}e^{2\pi i(s-t)/(Q-1)}.$$
In particular, if $s-t\neq0 {\rm mod} (Q-1)$, then $a_{st}=0$. Thus
we have

\begin{equation*}
H(p)=\begin{cases}
 1+
a_{20}(p^2+\bar{p}^2)+a_{11}|p|^2+\mathcal{O}(|p|^4),\ \  &\text{ if } Q=3,\\
1+a_{11}|p|^2+\mathcal{O}(|p|^3),\ \ &\text{ if } Q=4,\\
1+a_{11}|p|^2+\mathcal{O}(|p|^4), \ \ &\text{ if } Q\geq 5.
\end{cases}
\end{equation*}

To compute the Hausdorff dimension of $\partial B_p$, we need the
following result (See \cite{Fal}, Theorem 9.1, Propositions 9.6 and
9.7)

\begin{thm}[{Falconer, \cite{Fal}}]\label{IFS} Let $S_1,\cdots,S_m$
be contractive maps on a closed subset $D$ of $\mathbb{R}^n$ such
that $|S_i(x)-S_i(y)|\leq c_i|x-y|$ with $c_i<1$.  Then

$(1)$ There there exists a unique non-empty compact set $J$ such
that $J=\bigcup_{\,i=1}^{\,m} S_i(J)$.

$(2)$ The Hausdorff dimension $H(J)$ of $J$ satisfies $H(J)\leq s$,
where $\sum_{i=1}^m c_i^s=1$.

$(3)$ If we require further $|S_i(x)-S_i(y)|\geq b_i|x-y|$  for
$i=1,\cdots,m$, then $H(J)\geq \widetilde{s}$, where $\sum_{i=1}^m
b_i^{\widetilde{s}}=1$.
\end{thm}

Now, we have

\begin{lema} For any $p\in \mathcal{V}$, the Hausdorff dimension $D=H(p)$ of $\partial B_p$ is determined
by the following equation
\begin{equation}\label{O(1)=}
\sum_{z\in\,\textup{Fix}(f_p^{\circ n})\cap \partial B_p}
|(f_p^{\circ n})'(z)|^{-D}=\mathcal{O}(1).
\end{equation}
\end{lema}
\begin{proof} Let $w_p\in \partial B_p$ be the landing point of the zero external
ray of $f_p$. We can split $w_p$ into two point $w_p^+$ and $w_p^-$
and view $\partial B_p$ as a closed segment with extreme points
$w_p^+$ and $w_p^-$. The map $f^{\circ n}_{p}:\partial
B_p\rightarrow
\partial B_p$ has $Q^n$ inverse branches, say $S_1,\cdots,S_{Q^n}$,
each maps $\partial B_p$ to a closed segment such that their images
are  in anticlockwise order. Moreover, $\partial B_p=\bigcup
S_i(\partial B_p)$. In particular, both $S_1(\partial B_p)$ and
$S_{Q^n}(\partial B_p)$ contain  $w_p$ as an end point, for
$1<j<Q^n$, $S_j(\partial B_p)$ contains exactly one fixed point of
$f^{\circ n}_{p}$.
By  Koebe distortion theorem and the fact that $\partial B_p$ is a
quasicircle, there exist two constants $C_1, C_2$ both independent
of $n$, such that
$$\frac{C_1}{|(f_p^{\circ
n})'(\zeta)|}\leq
\frac{|S_i(x)-S_i(y)|}{|x-y|}\leq\frac{C_2}{|(f_p^{\circ
n})'(\zeta)|},~\forall\,1\leq i\leq Q^n,\  x,y\in S_i(\partial
B_p),$$ where $\zeta$ is the unique fixed point of $f^{\circ n}_{p}$
in $S_i(\partial B_p)$.

By Theorem \ref{IFS}, we have $s_1\leq D\leq s_2$, where $\sum_\zeta
C_j^{s_j}|(f_p^{\circ n})'(\zeta)|^{-s_j}=1, \ j=1,2.$ It turns out
that when $n$ is large, the sum $\sum_\zeta |(f_p^{\circ
n})'(\zeta)|^{-D}$ is a number between $C_1^{-D}$ and $C_2^{-D}$ and
 $$\sum_{z\in\,\text{Fix}(f_p^{\circ n})\cap
\partial B_p} |(f_p^{\circ
n})'(z)|^{-D}=\sum_\zeta|(f_p^{\circ n})'(\zeta)|^{-D}-|(f_p^{\circ
n})'(w_p)|^{-D}=\mathcal{O}(1).$$ The proof is completed.
\end{proof}

\noindent\textit{Proof of Theorem \ref{main}.} Note that when $p=0$,
the Julia set $J(f_p)$ is the unit circle which can be parameterized
by $z(t)=e^{2\pi it}$ such that
\begin{equation}\label{f_p(z_t)}
f_p(z(t))=z(Qt).
\end{equation}
For small $p\neq 0$, the restriction $f_p:\partial
B_p\rightarrow\partial B_p$ is a covering map with degree $Q$. Then
$\partial B_p$ can parameterized such that (\ref{f_p(z_t)}) holds
since $\partial B_p$ is homeomorphic to the unit circle. By Lemma
\ref{holo-motion}, we know that the point $z(t)$ on $\partial B_p$
moves holomorphically on $p$. This means that, in a neighborhood of
0, we can expand $z(t)$ by
\begin{equation}\label{z_t}
z(t)=e^{2\pi it}(1+p\,U_1(t)+p^2\,U_2(t)+\mathcal{O}(p^3)),
\end{equation}
where $U_m(t)$ satisfies $U_m(t+1)=U_m(t)$ for $m\geq 1$.
Substituting (\ref{z_t}) into (\ref{f_p(z_t)}), then comparing the
same order in $p$, we have the following equations
\begin{eqnarray}
U_1(Qt)-QU_1(t) &=& e^{-2\pi i (2Q)t},\label{equation_1}\\
U_2(Qt)-QU_2(t) &=& \DF{Q(Q-1)}{2}U_1^2(t)-e^{-2\pi i
(2Q)t}QU_1(t).\label{equation_2}
\end{eqnarray}
It is easy to verify the linear functional equation $
\phi(Qt)-Q\phi(t) = e^{-2\pi i t} $ has the solution
\begin{equation}
\phi(t) = -\DF{1}{Q}\sum_{l=0}^{\infty}Q^{-l}e^{-2\pi i Q^l t}.
\end{equation}
Hence we can solve the equations (\ref{equation_1}) and
(\ref{equation_2}) by
\begin{eqnarray}
U_1(t) &=& \phi(2Qt),\label{solution_1}\\
U_2(t) &=&
\DF{Q(Q-1)}{2}\sum_{l_1,\,l_2=1}^{\infty}Q^{-(l_1+l_2)}\phi(2(Q^{l_1}+Q^{l_2})t)
+Q\sum_{l=1}^{\infty}Q^{-l}\phi(2(Q^l+Q)t).\label{solution_2}
\end{eqnarray}

Actually, the higher order terms $U_m(t)$ with $m\geq 3$ can also be
calculated by induction. But it will be extremely complicated.

Notice that the fixed point of $f_p^{\circ n}$ forms the following
set
\begin{equation}
\text{Fix}(f_p^{\circ n})\cap \partial
B_p=\{z(t_j):\,t_j=j/(Q^n-1),j=0,1,\cdots,Q^n-2\}.
\end{equation}
 Following\cite{WBK}, it is
convenient to introduce the \textit{average notation}
\begin{equation}
\langle G(t)\rangle_n=\DF{1}{Q^n-1}\sum_{j=0}^{Q^n-2}G(t_j).
\end{equation}
A very useful property of this average is
\begin{equation}\label{prpt_of_average}
{\langle e^{2\pi i m t}\rangle_n}=\left\{
\begin{array}{ll}
 1~~~~~\text{if}~m\equiv 0~\text{mod}~Q^n-1,\\
 0~~~~~\text{otherwise.}
\end{array}
\right.
\end{equation}
By the fact that
\begin{equation}
(f_p^{\circ
n})'(z(t_j))=\prod_{m=0}^{n-1}f_p'(z(Q^mt_j))=Q^n\prod_{m=0}^{n-1}\left(z^{Q-1}(Q^mt_j)-\DF{p}{z^{Q+1}(Q^mt_j)}\right),
\end{equation}
we can write (\ref{O(1)=}) as
\begin{equation}\label{W_n(D)}
\mathcal{O}(1)=Q^{-nD}(Q^n-1)\left\langle
\prod_{m=0}^{n-1}\left|z^{Q-1}(Q^mt)-\DF{p}{z^{Q+1}(Q^mt)}\right|^{-D}\right\rangle_n.
\end{equation}

The calculation in Appendix shows that for all sufficiently large
$n$, we have
\begin{equation}\label{result}
\mathcal{O}(1)=Q^{-nD}(Q^n-1)(1+D^2 n|p|^2+\mathcal{O}(np^3)).
\end{equation}
Fix some large $n$, when $p$ is small enough, we have
\begin{equation}\label{ff_result}
   \mathcal{O}(1)
 =\exp\Big(n(-(D-1)\log Q + D^2 |p|^2)\Big).
\end{equation}
This means that
\begin{equation}
D=1+\DF{|p|^2}{\log Q}+\mathcal{O}(|p|^3),
\end{equation}
which is the required formula in the main theorem. \hfill $\Box$

\section{Appendix}

This section will devote to prove (\ref{result}). Firstly, we do
some simplifications on notations. We use $z_m$, $U_{1,m}$ and
$U_{2,m}$ to denote $z(Q^m t)$, $U_1(Q^m t)$ and $U_2(Q^m t)$
respectively. Let $\sigma=e^{2\pi i t}$, by (\ref{z_t}), we have
\begin{equation}
\begin{split}
   &~ \left|z_m^{Q-1}-p/z_m^{Q+1}\right|=\left|(1+V_m)^{Q-1}-\sigma^{-2Q^{m+1}}p/(1+V_m)^{Q+1}\right|\\
 = &~ \left|1+(Q-1)V_m+(Q-1)(Q-2)V_m^2/2-\sigma^{-2Q^{m+1}}p\,(1-(Q+1)V_m)+\mathcal{O}(p^3)\right|,
\end{split}
\end{equation}
where $V_m=U_{1,m}\,p+U_{2,m}\,p^2+\mathcal{O}(p^3)$. So
\begin{equation}
\begin{split}
   &~ \left|z_m^{Q-1}-p/z_m^{Q+1}\right|^{-\frac{D}{2}}
      = \left|1+\left[(Q-1)U_{1,m}-\sigma^{-2Q^{m+1}}\right]p\right.\\
   &~ \left.+\left[(Q-1)(Q-2)U_{1,m}^2/2+(Q+1)\sigma^{-2Q^{m+1}}U_{1,m}+(Q-1)U_{2,m}\right]p^2+\mathcal{O}(p^3)\right|^{-\frac{D}{2}}\\
 = &~ \left|1-\frac{D}{2}A_m p+\frac{D}{8}B_m p^2+\mathcal{O}(p^3)\right|,
\end{split}
\end{equation}
where
\begin{eqnarray}
A_m &=& (Q-1)\,U_{1,m}-\sigma^{-2Q^{m+1}},\label{A_m}\\
B_m &=& (Q-1)\,(D(Q-1)+2)U_{1,m}^2-2(D(Q-1)+4Q)\,\sigma^{-2Q^{m+1}}U_{1,m}\label{B_m}\\
    & & -4(Q-1)\,U_{2,m}+(D+2)\,\sigma^{-4Q^{m+1}}. \nonumber
\end{eqnarray}
Then we have
\begin{equation}\label{formu}
\begin{split}
   &~ \left|z_m^{Q-1}-p/z_m^{Q+1}\right|^{-D}
      = \left(1-\frac{D}{2}A_m p+\frac{D}{8}B_m p^2\right)
        \left(1-\frac{D}{2}\overline{A}_m \overline{p}+\frac{D}{8}\overline{B}_m \overline{p}^2\right)+\mathcal{O}(p^3)\\
 = &~ 1-\frac{D}{2}(A_m p+\overline{A}_m \overline{p})+\frac{D^2}{4}|p|^2 A_m\overline{A}_m
      +\frac{D}{8}(B_m p^2+\overline{B}_m \overline{p}^2)+\mathcal{O}(p^3).
\end{split}
\end{equation}

\begin{lema}\label{integer1}
Let $u,v\in\mathbb{N}$, for any large $n$, then $(1)$
$2Q^v/(Q^n-1)\not\equiv 0~\textup{mod}~1$; $(2)$
$2Q^v(Q^u+1)/(Q^n-1)\not\equiv 0~\textup{mod}~1$.
\end{lema}
\begin{proof}
Since $(Q,Q^n-1)=1$, it follows that $(Q^v,Q^n-1)=1$. This means
that $2Q^v/(Q^n-1)$ can not be an integer since $0< 2<Q^n-1$ for
large $n$.

For the second assertion, suppose that $u=ns+r$ for $0\leq r<n$,
then
\begin{equation*}
\frac{2Q^v(Q^u+1)}{Q^n-1}\equiv 2Q^v\,\frac{Q^{ns}(Q^r-1)+2}{Q^n-1}
\equiv \frac{2Q^v(Q^r+1)}{Q^n-1} ~\text{mod}~1.
\end{equation*}
Since $(Q^v,Q^n-1)=1$, this means that $2Q^v(Q^u+1)/(Q^n-1)$ can not
be an integer because $0< 2(Q^r+1)<Q^n-1$ for large $n$.
\end{proof}

By Lemma \ref{integer1}, combine the average property of
(\ref{prpt_of_average}), it is easy to verify the following
\begin{cor}\label{no_p^k}
$\langle A_m\rangle_n=0$, $\langle A_mA_k\rangle_n=0$ and $\langle
B_m\rangle_n=0$ for $0\leq m,k\leq n-1$.
\end{cor}

Now, by (\ref{formu}), we have
\begin{equation}\label{f_result_1}
\begin{split}
   &~
   \left\langle\prod_{m=0}^{n-1}\left|z_m^{Q-1}-p/z_m^{Q+1}\right|^{-D}\right\rangle_n
      =1+\frac{D^2}{4}|p|^2 \sum_{m,\,k=0}^{n-1}\langle
      A_m\overline{A}_k\rangle_n+\mathcal{O}(|p|^3).
\end{split}
\end{equation}

Substituting (\ref{A_m}) to $A_m\overline{A}_k$, we have
\begin{equation}\label{A_m_A_k_expand}
\begin{split}
      \langle A_m\overline{A}_k\rangle_n
   =&~ (Q-1)^2\langle U_{1,m}\overline{U}_{1,k}\rangle_n+\langle \sigma^{2Q(Q^k-Q^m)}\rangle_n\\
    & -(Q-1)\,(\langle \sigma^{-2Q^{m+1}}\overline{U}_{1,k}\rangle_n+\langle\sigma^{2Q^{k+1}}U_{1,m}\rangle_n).
\end{split}
\end{equation}

\begin{lema}\label{integer}
Let $u\in\mathbb{N}$, $(Q^u-1)/(Q^n-1)$ is an integer if and only if
$u=ns$ for some $s\in\mathbb{N}$.
\end{lema}
\begin{proof}
The ``if" part is trivial, we only prove the ``only if" part.
Suppose that $u=ns+r$ for $0\leq r<n$, according to the assumption,
we have
\begin{equation*}
\frac{Q^u-1}{Q^n-1}\equiv \frac{Q^{ns}(Q^r-1)}{Q^n-1}~\text{mod}~1.
\end{equation*}
Since $(Q^{ns},Q^n-1)=1$, we conclude that $(Q^u-1)/(Q^n-1)$ is an
integer  if and only if $(Q^r-1)/(Q^n-1)$ is an integer, namely
$r=0$.
\end{proof}

From Lemma \ref{integer} and the property (\ref{prpt_of_average}) of
average notation, it follows that
\begin{equation}\label{U(1,m)BarU(1,k)}
\begin{split}
  &~ \sum_{m,\,k=0}^{n-1}\langle U_{1,m}\overline{U}_{1,k}\rangle_n=
      \DF{1}{Q^2}\sum_{m,\,k=0}^{n-1}\,\sum_{l_1,\,l_2=0}^{\infty}\DF{1}{Q^{\,l_1+l_2}}
     \left\langle \sigma^{-2Q(Q^{l_1+m}-Q^{l_2+k})}\right\rangle_n\\
 =&~ \DF{1}{Q^2}\sum_{m,\,k=0}^{n-1}\left(\sum_{l_1+m = l_2+k}\DF{1}{Q^{\,l_1+l_2}}
                     +\sum_{v\neq 0}~\sum_{l_1+m = l_2+k+nv}\DF{1}{Q^{\,l_1+l_2}}\right)\\
 =&~ \DF{1}{Q^2}\left(\sum_{m=0}^{n-1}\,\sum_{k=0}^{m}\,\sum_{l_1=0}^{\infty}\DF{1}{Q^{\,2l_1+m-k}}
                     +\sum_{k=0}^{n-1}\,\sum_{m=0}^{k-1}\,\sum_{l_2=0}^{\infty}\DF{1}{Q^{\,2l_2+k-m}}\right)\\
  &~+ \DF{1}{Q^2}\left(\sum_{v=1}^{+\infty}\,\sum_{m,\,k=0}^{n-1}\,\sum_{l_2=0}^{\infty}\DF{1}{Q^{\,2l_2+k-m+nv}}
                     +\sum_{v=-1}^{-\infty}\,\sum_{m,\,k=0}^{n-1}\,\sum_{l_1=0}^{\infty}\DF{1}{Q^{\,2l_1+m-k-nv}}\right)\\
 =&~ \DF{1}{Q^2-1}\left(\DF{Q+1}{Q-1}n+\mathcal{O}(1)\right)+\mathcal{O}(1)=\DF{n}{(Q-1)^2}+\mathcal{O}(1).
\end{split}
\end{equation}
Here we have used the following formulas
\begin{eqnarray}
\sum_{m=0}^{n-1}\,\sum_{k=0}^{m}\,\DF{1}{Q^{m-k}} &=& \DF{n Q}{Q-1}-\DF{Q-Q^{-(n-1)}}{(Q-1)^2}=\DF{n Q}{Q-1}+\mathcal{O}(1),\label{sum_large}\\
\sum_{m=0}^{n-1}\,\sum_{k=0}^{m-1}\,\DF{1}{Q^{m-k}} &=&
\DF{n}{Q-1}-\DF{Q-Q^{-(n-1)}}{(Q-1)^2}=\DF{n}{Q-1}+\mathcal{O}(1).\label{sum_small}
\end{eqnarray}

The calculation in (\ref{U(1,m)BarU(1,k)}) shows that the sum of the
case for $l_1+m\neq l_2+k$ is bounded above by a constant depending
only on $Q$ when $n$ tends to $\infty$, which we marked by
$\mathcal{O}(1)$. This observation is important in the following
similar calculations. Namely, the main ingredients of the result is
derived from the case for $l_1+m = l_2+k$.

Similar to the calculation in (\ref{U(1,m)BarU(1,k)}), we have
\begin{equation}\label{sigma_k,m}
\sum_{m,\,k=0}^{n-1}\left\langle\sigma^{2Q(Q^k-Q^m)}\right\rangle_n=n+\mathcal{O}(1)
\end{equation}
and
\begin{equation}\label{sigma_u}
\begin{split}
  &~ \sum_{m,\,k=0}^{n-1}\left(\left\langle \sigma^{-2Q^{m+1}}\overline{U}_{1,k}\right\rangle_n
           +\left\langle\sigma^{2Q^{k+1}}U_{1,m}\right\rangle_n\right)=
      -\DF{2}{Q}\sum_{m,\,k=0}^{n-1}\,\sum_{l=0}^{\infty}\DF{1}{Q^{\,l}}
     \left\langle \sigma^{-2Q(Q^{m}-Q^{k+l})}\right\rangle_n\\
 =&~ -\DF{2}{Q}\sum_{m=0}^{n-1}\,\sum_{k=0}^{m}\DF{1}{Q^{\,m-k}}+\mathcal{O}(1)=-\DF{2n}{Q-1}+\mathcal{O}(1).
\end{split}
\end{equation}
Combine (\ref{A_m_A_k_expand}), (\ref{U(1,m)BarU(1,k)}),
(\ref{sigma_k,m}) and (\ref{sigma_u}), we have
\begin{equation}\label{A_m_A_k}
\sum_{m,\,k=0}^{n-1}\langle
A_{m}\overline{A}_{k}\rangle_n=n+n+2n+\mathcal{O}(1)=4n+\mathcal{O}(1).
\end{equation}
From (\ref{f_result_1}), this means that
\begin{equation}
\left\langle\prod_{m=0}^{n-1}\left|z_m^{Q-1}-p/z_m^{Q+1}\right|^{-D}\right\rangle_n
=1+D^2 n|p|^2+\mathcal{O}(np^3).
\end{equation}
The proof of (\ref{result}) is completed.


\bibliography{References}

\begin{thebibliography}{99}
\addcontentsline{toc}{section}{References}
\addtolength{\itemsep}{-1.5ex} 

\bibitem{CDM}P. Collet, R. Dobbertin and P. Moussa, Multifractal
analysis of nearly circular Julia set and thermodynamical formalism.
Ann. Inst. H Poincar\'{e}, \textbf{56} (1992), 91-122.
\bibitem{Dev}  R. Devaney,  Intertwined Internal Rays in Julia Sets of Rational
Maps, Fund. Math. \textbf{206} (2009), 139-159.
\bibitem{DG} R. Devaney and A. Garijo,  Julia Sets Converging to the
Unit Disk.  Proc. AMS, \textbf{136} (2008), 981-988.
\bibitem{DLU}R. Devaney, D. Look and D. Uminsky,
The Escape Trichotomy for Singularly Perturbed Rational Maps ,
Indiana University Mathematics Journal \textbf{54} (2005),
1621-1634.
\bibitem{Fal}K. J. Falconer, {Fractal geometry: mathematical foundations and
applications}. John Wiley \& Sons, 1990.
\bibitem{MSS}R. Ma\~{n}\'{e}, P. Sad and D. Sullivan, On the dynamics of rational maps. Ann. Sci. \'{E}cole Norm. Sup. (4) \textbf{16} (1983), 193-217.
\bibitem{McM}C. McMullen, Automorphisms of rational maps. in
\textit{Holomorphic Functions and Moduli I}, Math. Sci. Res. Inst.
Publ. \textbf{10}, Springer, 1988.
\bibitem{McM1} C. McMullen,  Complex Dynamics and Renormalization, Ann. of
Math. Studies 135, Princeton Univ. Press, Princeton, NJ, 1994.
\bibitem{Osb}A. Osbaldestin, $1/s$--expansion for generalized
dimensions in a hierarchical $s$--state Potts model. J. Phys. A:
Math. Gen. \textbf{28} (1995), 5951-5962.
\bibitem{Rue}D. Ruelle, Repellers for real analytic maps. Ergodic Theory Dynamical Systems, \textbf{2} (1982), 99-107.
\bibitem{St}N. Steinmatz, On the dynamics of McMullen family.
Conformal Geometry and Dynamics. \textbf{10} (2006) 159-183.
\bibitem{WQY}  W. Qiu, X. Wang, Y. Yin.  Dynamics of  McMullen maps.
Advances in Mathematics. 229(2012), 2525-2577.
\bibitem{WBK}M. Widom, D. Bensimon, L. P. Kadanoff and S. J. Shenker, Strange objects in the complex plane. Journal of Statistical Physics, \textbf{32} (1983), 443-454.


\end{thebibliography}

Fei YANG

\vskip0.1cm
\textsc{School of Mathematical Sciences, Fudan University,\\
Shanghai, 200433, P.R.China}

\vskip0.2cm \textit{E-mail address}:
\textsf{yangfei\rule[-2pt]{0.2cm}{0.5pt}math@163.com}

\vskip0.4cm

Xiaoguang WANG

\vskip0.1cm \textsc{School of Mathematical Sciences, Fudan University,\\
Shanghai, 200433, P.R.China}

\vskip0.2cm \textit{E-mail address}: \textsf{wxg688@163.com}

\end{document}